\documentclass{amsart}
\usepackage{graphicx}
\usepackage{amssymb, amsfonts, amsthm}

\newtheorem{theorem}{Theorem}[section]
\newtheorem{lemma}[theorem]{Lemma}
\newtheorem{corollary}[theorem]{Corollary}
\newtheorem{proposition}[theorem]{Proposition}

\theoremstyle{definition}
\newtheorem{definition}[theorem]{Definition}

\newtheorem{conjecture}[theorem]{Conjecture}
\newtheorem{problem}[theorem]{Problem}

\theoremstyle{remark}
\newtheorem{remark}[theorem]{Remark}

\numberwithin{equation}{section}

\begin{document}

\title{Segment Number of Knots}

\author{Makoto Ozawa}
\address{Department of Natural Sciences, Faculty of Arts and Sciences, Komazawa University, 1-23-1 Komazawa, Setagaya-ku, Tokyo, 154-8525, Japan}
\email{w3c@komazawa-u.ac.jp}




\begin{abstract}
We introduce a new numerical knot invariant, termed the \textit{segment number}, which is derived from partitioned knot diagrams subject to specific over/under-crossing constraints. We prove that a knot is non-trivial if and only if its segment number is at least 3. Furthermore, we investigate the structural properties of the directed graph associated with a minimal segment number presentation. Specifically, we show that for any minimal presentation, the underlying graph is connected and cannot be a path. Finally, we discuss the relationship between the segment number and the bridge number, providing bounds and conjectures for future study.
We also conjecture that the bridge number $b(K)$ provides a lower bound for the segment number.
\end{abstract}

\maketitle

\section{Introduction}

Knot diagrams are traditionally studied through various numerical invariants, such as the crossing number, bridge number, and unknotting number. These invariants quantify the complexity of a knot in terms of its projections onto a plane. In this paper, we propose a new invariant, the \emph{segment number} $s(K)$, defined by partitioning a knot diagram into a set of arcs (segments) that satisfy specific over/under-crossing constraints.\footnote{This invariant was referred to as the ``edge number'' in an earlier preprint (arXiv:0705.4348). We have adopted the term ``segment number'' to avoid confusion with the stick number.}

The motivation for this invariant arises from a structural analysis of why knots cannot be trivially unraveled. Consider a diagram partitioned into several segments. If the crossing information between any pair of segments is consistent (e.g., segment $A$ is always positioned over segment $B$), a precedence relationship is established. We hypothesize that a ``deadlock'' within these relationships is what prevents a knot from being trivial.

For instance, Figure \ref{trefoil} illustrates a standard diagram of the trefoil knot partitioned into three segments: $e_1, e_2$, and $e_3$. In this presentation, $e_1$ lies entirely over $e_2$, $e_2$ over $e_3$, and $e_3$ over $e_1$. This cyclic, ``rock-paper-scissors''-like relation creates a topological deadlock, the complexity of which is captured by the segment number.

\begin{figure}[htbp]
	\centering
    \begin{tabular}{cc}
         \includegraphics[trim=0mm 0mm 0mm 0mm, width=.3\linewidth]{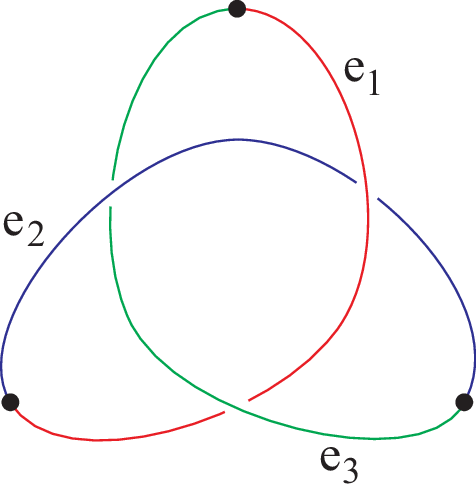} & 
         \includegraphics[trim=0mm 0mm 0mm 0mm, width=.33\linewidth]{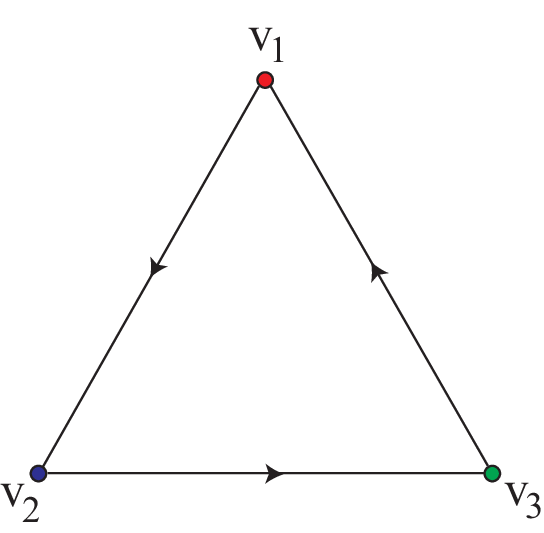}\\
         $D$ & $G(D)$
    \end{tabular}
	\caption{A 3-segment presentation $D$ of the trefoil knot and the digraph $G(D)$ corresponding to $D$}
	\label{trefoil}
\end{figure}

The remainder of this paper is organized as follows. We first provide a formal definition of the segment number and establish its fundamental properties. We then prove that $s(K) \ge 3$ for any non-trivial knot (Proposition \ref{non-trivial}) and analyze the topology of the associated directed graphs for minimal presentations (Theorem \ref{main}). Finally, we examine knots with $s(K)=3$ for specific crossing numbers and discuss the relationship between the segment number and the bridge number.

\section{Definition of segment Number}

Throughout this paper, we work within the piecewise linear category and study knots in the three-dimensional Euclidean space $\mathbb{R}^{3}$ or the 3-sphere $S^3$. For standard definitions and results in knot theory, we refer the reader to \cite{R, BZ, Kaw}.

Let $K$ be a knot and $D$ be a diagram of $K$ on the 2-sphere $S^{2}$. We partition $D$ into $n$ arcs $e_1, \ldots, e_n$ ($n \ge 1$) by choosing $n$ points on the diagram that are distinct from the crossings. We call such a diagram $D$ an \textit{$n$-partitioned diagram}.

\begin{definition}
An $n$-partitioned diagram $D$ is called an \textit{$n$-segment presentation} of $K$ if it satisfies the following conditions:
\begin{enumerate}
    \item Each segment $e_i$ contains no self-crossings.
    \item For any pair of distinct segments $e_i$ and $e_j$, exactly one of the following holds:
    \begin{enumerate}
        \item $e_i$ crosses \textit{over} $e_j$ at every crossing point between them.
        \item $e_i$ crosses \textit{under} $e_j$ at every crossing point between them.
        \item There are no crossings between $e_i$ and $e_j$.
    \end{enumerate}
\end{enumerate}
The \textit{segment number} $s(K)$ of a knot $K$ is the minimum integer $n$ such that $K$ admits an $n$-segment presentation.
\end{definition}

We can associate a directed graph (digraph) $G(D)$ with an $n$-segment presentation $D$ as follows:
\begin{enumerate}
    \item Assign a vertex $v_i$ in $G(D)$ to each segment $e_i$ of $D$.
    \item For any two segments $e_i$ and $e_j$ that intersect, we define a directed edge:
    \begin{itemize}
        \item from $v_i$ to $v_j$ if $e_i$ passes over $e_j$ at their crossings;
        \item from $v_j$ to $v_i$ if $e_j$ passes over $e_i$ at their crossings.
    \end{itemize}
\end{enumerate}

As shown in Figure \ref{trefoil}, the digraph $G(D)$ obtained from the 3-segment presentation of the trefoil knot forms an oriented 3-cycle.

\begin{remark}
The strict consistency required by condition (2) is essential. If mixed over/under crossings were permitted between a single pair of segments, any knot would admit a trivial ``2-segment'' presentation (see Figure \ref{trefoil2}).\footnote{This observation has been cited in various studies following the preprint's submission (e.g., \cite{AST, IT, BM, EHLN}) and has been further developed under terms such as ``2-colored diagrams'' and ``meander diagrams.'' While its origins can be traced back to Hotz (\cite{H}, 1960), it is also noted that the concept was known to Gauss.}
\end{remark}

\begin{figure}[htbp]
	\centering
	\includegraphics[trim=0mm 0mm 0mm 0mm, width=.6\linewidth]{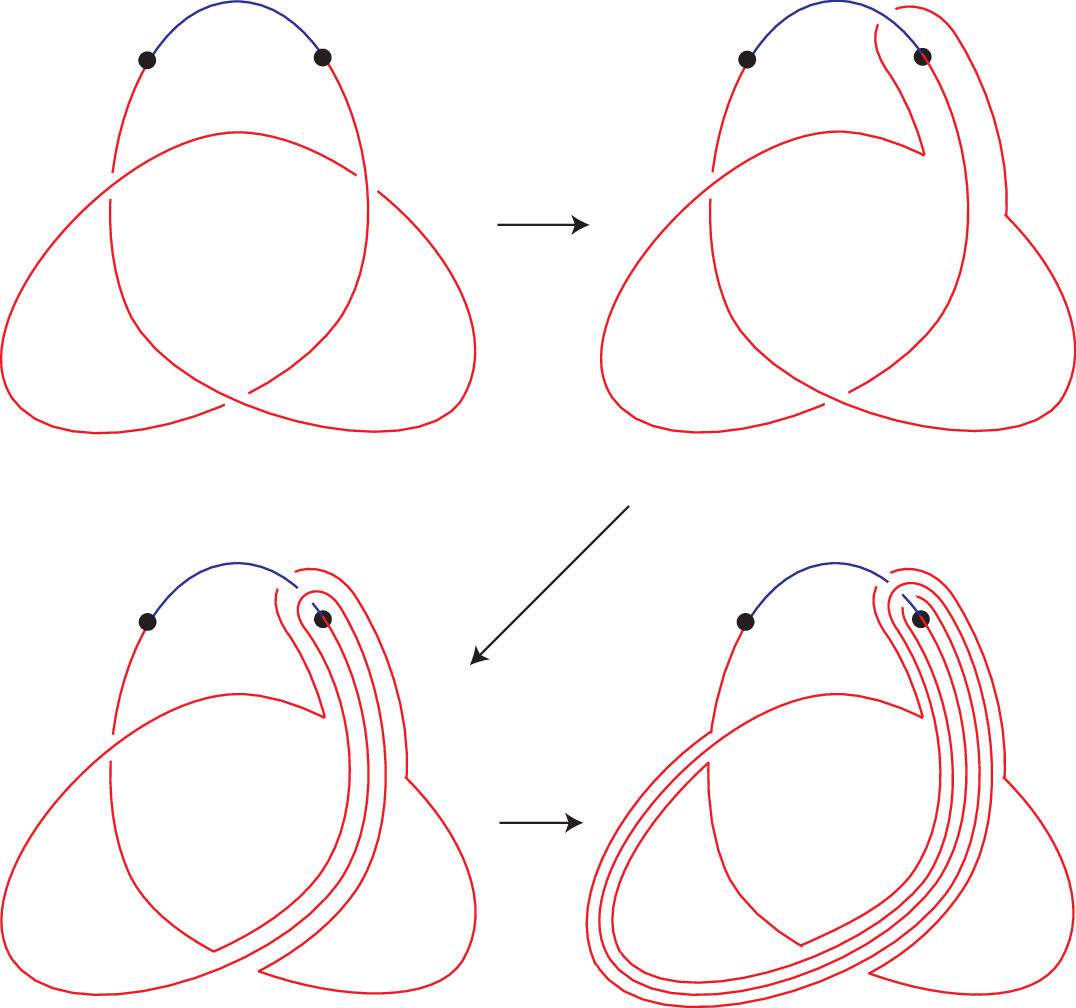}
	\caption{A 2-partitioned diagram that is NOT a 2-segment presentation because the crossing relations are mixed.}
	\label{trefoil2}
\end{figure}

\section{Main Results}

We begin by establishing the lower bound for the segment number of non-trivial knots.

\begin{proposition}\label{non-trivial}
A knot $K$ is non-trivial if and only if $s(K) \ge 3$.
\end{proposition}

\begin{proof}
($\Rightarrow$) Suppose $s(K)=1$, and let $D$ be a 1-segment presentation. By condition (1) of the definition, $D$ has no self-crossings. Thus, $D$ is a simple closed curve in the plane with no crossings, which implies $K$ is the trivial knot. 

Next, suppose $s(K)=2$ and let $D$ be a 2-segment presentation with segments $e_1$ and $e_2$. Without loss of generality, assume $e_1$ passes over $e_2$ at every crossing. This configuration corresponds to a 1-bridge presentation where $e_1$ acts as the over-bridge and $e_2$ as the under-bridge. Since a knot with bridge number 1 is necessarily trivial, $s(K)=2$ also implies $K$ is the trivial knot.

($\Leftarrow$) If $K$ is trivial, it admits a diagram with no crossings, which can be viewed as a 1-partitioned diagram satisfying the required conditions; hence $s(K)=1$. By contraposition, if $K$ is non-trivial, then $s(K) \ge 3$.
\end{proof}

When the segment number is minimal ($s(K)=3$), the structure of the associated digraph is remarkably rigid.

\begin{proposition}\label{3-segment}
Let $K$ be a knot with $s(K)=3$, and let $D$ be a minimal 3-segment presentation. Then $G(D)$ is an oriented 3-cycle.
\end{proposition}

\begin{proof}
Let $D$ be a 3-segment presentation with segments $e_1, e_2, e_3$. Suppose $G(D)$ does not form an oriented 3-cycle. Since $G(D)$ has 3 vertices, it must be acyclic (i.e., a transitive tournament or a subgraph thereof). Without loss of generality, assume the vertex ordering is $v_1 \to v_2 \to v_3$. This implies:
\begin{itemize}
    \item $e_1$ is over $e_2$,
    \item $e_2$ is over $e_3$,
    \item $e_1$ is over $e_3$ (if they intersect).
\end{itemize}
In this case, $D$ is a \textit{descending diagram} with respect to the ordering $(e_1, e_2, e_3)$. It is a well-known result in knot theory that any knot admitting a descending diagram is trivial. This contradicts the assumption that $s(K)=3$, as the unknot has $s(\text{Unknot})=1$. Therefore, $G(D)$ must contain a cycle, which for $n=3$ must be an oriented 3-cycle.
\end{proof}

For the general case, we provide the following structural theorem regarding the digraph of a minimal presentation.

\begin{theorem}\label{main}
Let $D$ be a minimal $n$-segment presentation of a non-trivial knot $K$. Then:
\begin{enumerate}
    \item $G(D)$ is connected.
    \item $G(D)$ is not a path.
\end{enumerate}
\end{theorem}

To prove Theorem \ref{main}, we employ the following lemma regarding the local structure of $G(D)$.

\begin{lemma}\label{neighbourhood}
Let $D$ be a minimal $n$-segment presentation of a knot $K$. Then:
\begin{enumerate}
    \item For any two consecutive segments $e_i$ and $e_{i+1}$, the vertices $v_i$ and $v_{i+1}$ cannot both be sources or both be sinks in the subgraph of $G(D)$ formed by their relations with all other segments $e_j$ ($j \ne i, i+1$).
    \item For any two consecutive but non-crossing segments $e_i$ and $e_{i+1}$, there must exist a segment $e_k$ that crosses both, such that the crossing orientations are consistent (forming a directed path $v_i \to v_k \to v_{i+1}$ or $v_{i+1} \to v_k \to v_i$).
\end{enumerate}
\end{lemma}

\begin{proof}
(1) Suppose consecutive segments $e_i$ and $e_{i+1}$ are such that $v_i$ and $v_{i+1}$ are both sources (or both sinks) relative to all segments they intersect. This implies their union $E = e_i \cup e_{i+1}$ is entirely over (or entirely under) all other segments. Since $e_i$ and $e_{i+1}$ share only an endpoint and do not cross each other, their union $E$ can be isotoped into a single segment $e'$ without self-crossings. This yields an $(n-1)$-segment presentation, contradicting the minimality of $s(K)$. (See Figure \ref{reduce1}).

\begin{figure}[htbp]
	\centering
	\includegraphics[trim=0mm 0mm 0mm 0mm, width=.7\linewidth]{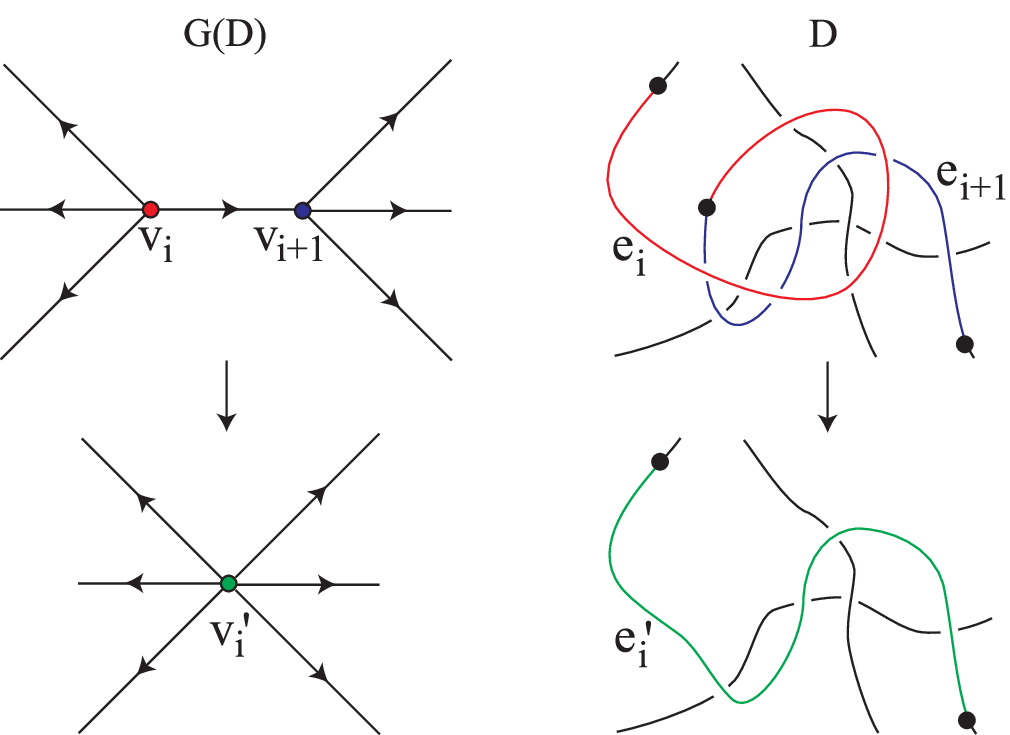}
	\caption{Reducing the segment number by merging $e_i$ and $e_{i+1}$.}
	\label{reduce1}
\end{figure}

(2) Similarly, if the crossing relationships with all neighboring segments do not force $e_i$ and $e_{i+1}$ to remain distinct, they can be merged into a single segment $e'$ without violating the definition of a segment presentation. This again contradicts the minimality of $s(K)$. (See Figure \ref{reduce2}).

\begin{figure}[htbp]
	\centering
	\includegraphics[trim=0mm 0mm 0mm 0mm, width=.7\linewidth]{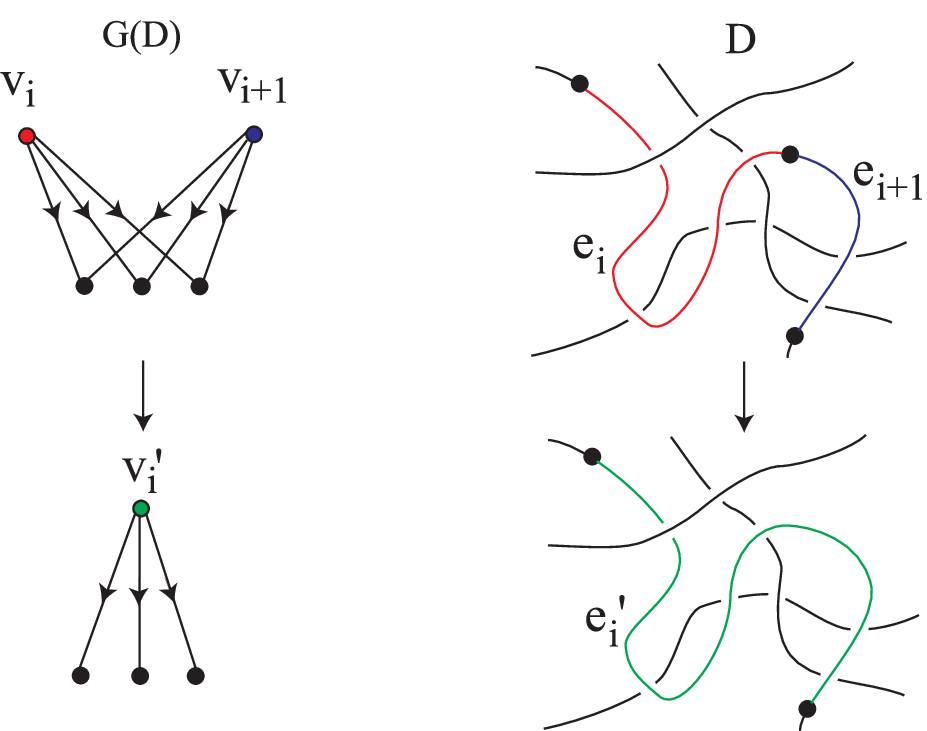}
	\caption{Merging segments when orientation constraints are not violated.}
	\label{reduce2}
\end{figure}
\end{proof}

Lemma \ref{neighbourhood} (2) implies the following corollary.

\begin{corollary}\label{distance}
In a minimal $n$-segment presentation $D$, the distance in $G(D)$ between vertices $v_i$ and $v_{i+1}$ corresponding to consecutive segments in the diagram is at most 2.
\end{corollary}

\begin{proof}[Proof of Theorem \ref{main}]
(1) Suppose $G(D)$ is disconnected. The set of segments then partitions into two disjoint sets $E_A$ and $E_B$ with no crossings between them. Since the knot $K$ is a connected curve, there must exist consecutive segments $e_i \in E_A$ and $e_{i+1} \in E_B$. By Corollary \ref{distance}, the distance between $v_i$ and $v_{i+1}$ is at most 2, which requires a path of length at most 2 in $G(D)$ connecting them. This contradicts the assumption that $E_A$ and $E_B$ are in different components.

(2) Suppose $G(D)$ is a path. By Proposition \ref{non-trivial}, $G(D)$ must have at least three vertices. If the path has exactly three vertices, Proposition \ref{3-segment} is violated (as it must be a cycle). If it has more than three vertices, the distance constraints between various consecutive segments $v_i, v_{i+1}$ imposed by Corollary \ref{distance} eventually force a cycle or a contradiction in the path's linear structure.
\end{proof}

We also verify that the segment number is independent of the crossing number modulo 6 for small segment numbers.

\begin{proposition}\label{exist}
For any $n \ge 3$ such that $n \not\equiv 5 \pmod{6}$, there exists a knot $K$ with crossing number $c(K)=n$ and $s(K)=3$.
\end{proposition}

\begin{proof}
The required knots can be constructed as alternating diagrams as shown in Figure \ref{6m}. These diagrams follow the pattern $n \equiv 0, 1, 2, 3, 4 \pmod{6}$. The minimality of the crossing number for these alternating diagrams is guaranteed by results \cite{Kau, Mur, Thi}.
\begin{figure}[htbp]
	\centering
	\begin{tabular}{ccc}
	\includegraphics[trim=0mm 0mm 0mm 0mm, width=.17\linewidth]{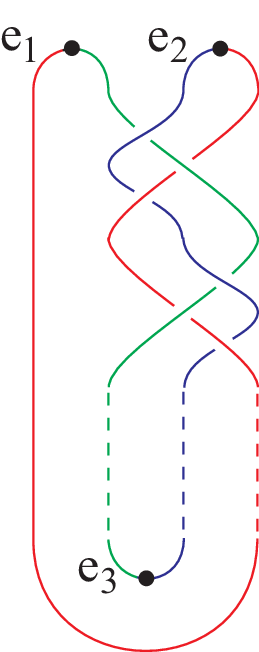}&
	\includegraphics[trim=0mm 0mm 0mm 0mm, width=.19\linewidth]{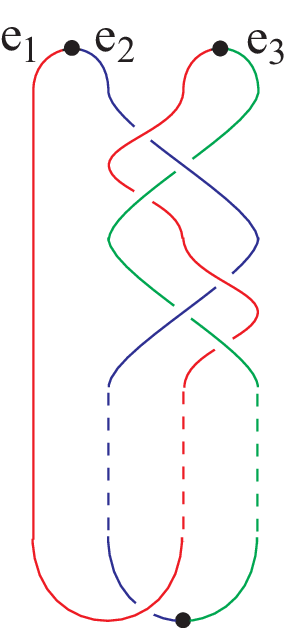}&
	\includegraphics[trim=0mm 0mm 0mm 0mm, width=.24\linewidth]{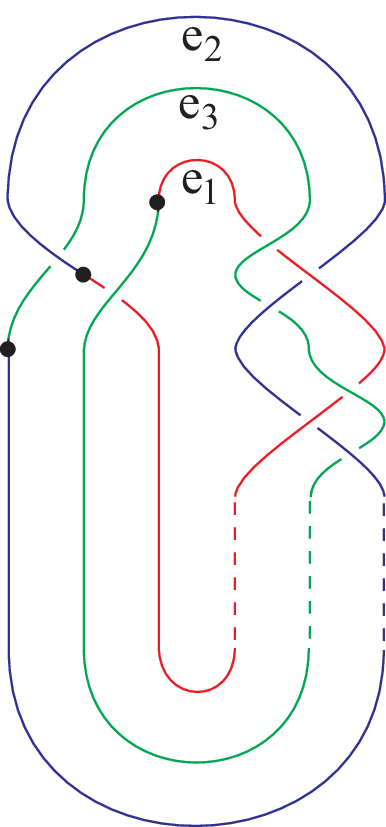}\\
	$n\equiv 0\mod 6$ & $n\equiv 1\mod 6$ & $n\equiv 2\mod 6$ \\
	\includegraphics[trim=0mm 0mm 0mm 0mm, width=.20\linewidth]{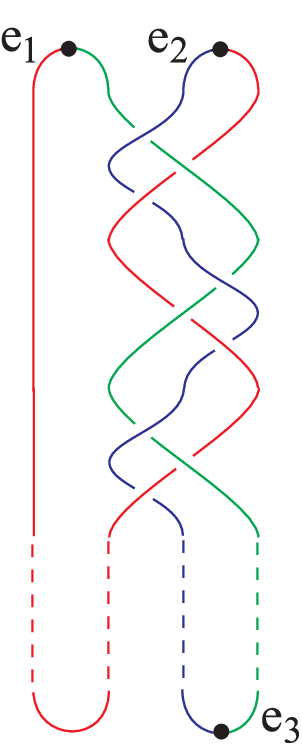}&
	\includegraphics[trim=0mm 0mm 0mm 0mm, width=.19\linewidth]{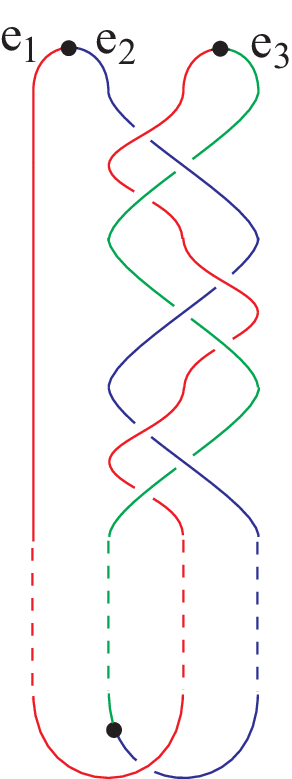}&
	\\
	$n\equiv 3\mod 6$ & $n\equiv 4\mod 6$ & \\
	\end{tabular}
	\caption{Examples of knots with $s(K)=3$ and varying crossing numbers.}
	\label{6m}
\end{figure}
\end{proof}

\section{Relationship between Segment Number and Bridge Number}

In this final section, we investigate the relationship between the segment number $s(K)$ and the bridge number $b(K)$.

\subsection{Upper Bound: $s(K) \le 2b(K)$}

\begin{proposition}
    For any knot $K$, the segment number satisfies $s(K) \le 2b(K)$.
\end{proposition}

\begin{proof}
Consider a bridge presentation of $K$ with $b(K)$ bridges, which consists of $b(K)$ over-bridges and $b(K)$ under-bridges. By treating each over-bridge and each under-bridge as an individual segment, we naturally obtain a $2b(K)$-segment presentation. Note that the resulting configuration forms a bipartite graph where the segments are joined at the bridge points. Thus, $s(K) \le 2b(K)$ follows.
\end{proof}

\subsection{Lower Bound Conjecture: $b(K) \le s(K)$}

Conversely, it is expected that the bridge number provides a lower bound for the segment number. We propose the following conjecture:

\begin{conjecture}\label{bridge}
For any knot $K$, $b(K) \le s(K)$.
\end{conjecture}

\subsection{Geometric Approach to Conjecture \ref{bridge}}

We sketch a geometric interpretation that supports the validity of Conjecture \ref{bridge}. Let $D$ be an $n$-segment presentation of a knot $K$ on a 2-sphere $S^2$ embedded in the 3-sphere $S^3$, where $n=s(K)$. This sphere $S^2$ divides $S^3$ into two 3-balls, $B_+$ and $B_-$. 

By an appropriate isotopy, we can push the neighborhoods of the partition points into $B_-$ and pull the interiors of the segments into $B_+$. In this configuration:
\begin{itemize}
    \item $K$ intersects $S^2$ at $2n$ points.
    \item $K \cap B_+$ consists of $n$ arcs properly embedded in $B_+$.
    \item $K \cap B_-$ consists of $n$ arcs properly embedded in $B_-$.
\end{itemize}

If the resulting $n$-string tangle $(B_+, K \cap B_+)$ is trivial (i.e., each arc is boundary-parallel or unknotted), the presentation would directly yield an $n$-bridge position, thereby implying $b(K) \le n = s(K)$. It is important to note, however, that the crossing constraints inherent in the definition of $s(K)$ (specifically Condition 2 of Definition 2.1) are more restrictive than those of standard bridge positions. This suggests that $s(K)$ may strictly exceed $b(K)$ for many knots.

\subsection{Future Work}

At present, it remains unknown whether there exists a knot $K$ such that $s(K) > 3$. While we anticipate the existence of such knots, the non-triviality of the segment number remains an open question. Consequently, we pose the following problem for future research:

\begin{problem}
Identify a knot $K$ with $s(K)=4$. We suggest the $5_1$ knot as a potential candidate for this property.
\end{problem}

\bibliographystyle{amsplain}

\end{document}